 \documentclass{birkmult}

  \usepackage{latexsym}
  \usepackage{amsmath}   
  \usepackage{amsfonts}
  \usepackage{amssymb}
  \usepackage{graphicx}
  \usepackage{amsthm}

  \newcommand{\C}{\mathbb{C}}
  \newcommand{\R}{\mathbb{R}}
  \newcommand{\HQ}{\ensuremath{\mathbb{H}}}
  \newcommand{\N}{\mathbb{N}}
  
  \newcommand{\qi}{\ensuremath{\mbox{\boldmath $i$}}}
  \newcommand{\sqi}{\ensuremath{\mbox{\boldmath $\scriptstyle i$}}}
  \newcommand{\qj}{\ensuremath{\mbox{\boldmath $j$}}}
  \newcommand{\sqj}{\ensuremath{\mbox{\boldmath $\scriptstyle j$}}}
  \newcommand{\qk}{\ensuremath{\mbox{\boldmath $k$}}}

  \newtheorem{thm}{Theorem}[section]
  \newtheorem{cor}[thm]{Corollary}

  \theoremstyle{definition}
  \newtheorem{defn}[thm]{Definition}
  \theoremstyle{remark}

  \newcommand{\be}{\begin{equation}}
  \newcommand{\ee}{\end{equation}} 
  \newcommand{\vect}[1]{\mathbf{#1}}
  \newcommand{\qftr}[1]{\overset{\triangleright}{#1}}
 
\begin{document}

\title[Quaternion FT \& Generalization]
 {Quaternion Fourier Transform on Quaternion Fields and Generalizations}
\author[E. Hitzer]{Eckhard M. S. Hitzer}

\address{%
Department of Applied Physics\\
University of Fukui\\
3-9-1 Bunkyo\\
910-8507 Fukui\\
Japan}

\email{hitzer@mech.fukui-u.ac.jp}

\thanks{I thank my family and FTHD organizer S.L. Eriksson.}



\date{October 19, 2006}
\dedicatory{Soli Deo Gloria}

\begin{abstract}
We treat the quaternionic
Fourier transform (QFT) applied to quaternion fields
and investigate QFT properties useful for applications. 
Different forms of the QFT
lead us to different Plancherel theorems. 
We relate the QFT computation for quaternion fields to the QFT of
real signals. We research the general linear ($GL$) transformation 
behavior of the QFT with matrices, Clifford geometric algebra and with examples.
We finally arrive
at wide-ranging non-commutative multivector FT generalizations of the QFT. Examples given are 
new volume-time and spacetime algebra Fourier transformations.
\end{abstract}

\maketitle



\section{Introduction}

This paper strives to deepen the understanding of the quaternionic
Fourier transform (QFT) applied to quaternion fields $f: \R^2 \rightarrow \HQ$, 
and not only to real signals $f: \R^2 \rightarrow \R$. 
We research QFT properties useful for applications
to partial differential equations, image processing and optimized
numerical implementations. We investigate how different forms of the QFT
allow to establish scalar and quaternion valued Plancherel theorems. 

We show systematically how to reduce the computation for quaternion fields to the case of
real signal computations, and on the other hand how results for real signals can be 
generalized to quaternion fields. 

The third major focus is on deriving the behavior of the QFT under 
$GL(\R^2)$ automorphisms. To do this we split the QFT appropriately,
and work with invariant techniques of Clifford geometric algebra \cite{HS:CAtoGC} 
to establish and
understand the automorphism behavior. Details are brought to light
by looking at the examples of stretches (dilations), reflections
and rotations. 
 
Together with isomorphisms (to Clifford subalgebras) we finally arrive
at wide-ranging generalizations of the QFT. These new non-commutative multivector
Fourier transforms operate
on functions from domain spaces $\R^{m,n}$ (with $m,n\in\N_0$) to Clifford algebras $ Cl_{m,n} $
or subalgebras thereof. 
To demonstrate the method, we work out 
generalizations to volume-time and to spacetime algebra 
Fourier transformations,
and provide some physical interpretation.

\subsection{Basic facts about Quaternions \label{sc:bquat}}

Gauss, Rodrigues and Hamilton \cite{TN:VisCA} invented the four-dimensional quaternion algebra 
$\HQ$ over $\R$ with three imaginary units $\qi$, $\qj$, $\qk$ and multiplication laws:

\be
 \qi \qj = -\qj \qi = \qk, \,\,
 \qj \qk = -\qk \qj = \qi, \,\,
 \qk \qi = -\qi \qk = \qj, \,\, 
 \qi^2=\qj^2=\qk^2=\qi \qj \qk = -1.
\label{eq:quat}
\end{equation}
Quaternions are isomorphic to the Clifford geometric algebra $Cl_{0,2}$ 
of $\R^{0,2}$, and to the even subalgebra $Cl_{3,0}^+$
of the Clifford geometric algebra $Cl_{3,0}$ of $\R^3$:
\be
  \label{eq:isom}
  \HQ\cong Cl_{0,2}\cong Cl_{3,0}^+.
\end{equation}
$Cl_{3,0}^+$ has, with an orthonormal basis $\{\vect{e}_1,\vect{e}_2,\vect{e}_3 \}$ 
of $\R^3$, the four dimensional basis 
\be
\{1,
\vect{e}_{32}=\vect{e}_3\vect{e}_2,
\vect{e}_{13}=\vect{e}_1\vect{e}_3,
\vect{e}_{21}=\vect{e}_2\vect{e}_1\}.
\end{equation}
Every quaternion 
\be
  q=q_r + q_i \qi + q_j \qj + q_k \qk \in \HQ, \quad 
  q_r,q_i, q_j, q_k \in \R
  \label{eq:aquat}
\end{equation}
has a \textit{quaternion conjugate} (corresponding to reversion in $Cl_{3,0}^+$)
\be
  \tilde{q} = q_r - q_i \qi - q_j \qj - q_k \qk,
\end{equation}
This leads to a \textit{norm} of $q\in\HQ$ defined as
\be
  |q| = \sqrt{q\tilde{q}} = \sqrt{q_r^2+q_i^2+q_j^2+q_k^2}.
\end{equation}

\subsection{Convenient rewriting of quaternions}

In some applications it proves convenient to replace $\qk$ with
$\qk=\qi\qj$ and write a quaternion as
\be
  q=q_r + \qi q_i  + q_j \qj + \qi q_k \qj,
  \label{eq:ijform}
\end{equation}
neatly keeping all $\qi$ to the left and all $\qj$ to the right of each term.
A second convenient form is the \textit{split}
\begin{gather}
  q = q_+ + q_-, \quad q_{\pm} = \frac{1}{2}(q\pm \qi q \qj).
  \label{eq:pmform}
\end{gather}
Explicitly in real components
$q_r,q_i, q_j, q_k \in \R$ using (\ref{eq:quat}) the split (\ref{eq:pmform})
produces:
\be
  q_{\pm} = \{q_r\pm q_k + \qi(q_i\mp q_j)\}\frac{1\pm \qk}{2}
       = \frac{1\pm \qk}{2} \{q_r\pm q_k + \qj(q_j\mp q_i)\}.
  \label{eq:qpm}
\end{equation}
The real scalar part $q_r$ (grade zero selection \cite{HS:CAtoGC} 
in Clifford geometric algebra) 
$$
  \langle q \rangle_0 = q_r
$$ 
leads to a cyclic multiplication symmetry
\be
  \langle qrs \rangle_0 = \langle rsq \rangle_0, \quad \forall q,r,s \in \HQ.
\label{eq:cycs}
\end{equation}

\subsection{Quaternion module}

For quaternion-valued functions $f,g: \R^2 \rightarrow \HQ$ we can define the 
quaternion-valued inner product
\be
  (f,g) = \int_{\R^2} f(\vect{x})\,\tilde{g}(\vect{x})\,d^2\vect{x} \,, 
  \quad \quad
  \text{ with } \quad d^2\vect{x} = dx dy,
\label{eq:intip}
\end{equation}
with symmetric real scalar part \cite{TB:thesis}
\be
  \langle f,g\rangle = \frac{1}{2}[(f,g)+(g,f)] 
  = \int_{\R^2} \langle f(\vect{x})\,\tilde{g}(\vect{x})\rangle_0 d^2\vect{x} \, .
\label{eq:intsp}
\end{equation}
Both (\ref{eq:intip}) and (\ref{eq:intsp}) lead to the  
$L^2(\R^2;\HQ)$-norm
\be
  \| f \| = \sqrt{(f,f)} 
          = \sqrt{\langle f,f\rangle} 
          = \int_{\R^2} |f(\vect{x})|^2 \,d^2\vect{x}
          \,\,.
\end{equation}
A \textit{quaternion module} $L^2(\R^2;\HQ)$ is then defined as
\be
  L^2(\R^2;\HQ ) = \{f | f:\R^2 \rightarrow \HQ, \|f\| < \infty\}.
\end{equation}

\section{The quaternion Fourier transform}

Before defining the quaternion Fourier transform (QFT), we briefly outline
its relationship with Clifford Fourier transformations. 

Brackx et al.~\cite{BDS:CA} extended the Fourier transform to multivector valued 
function-distributions in $Cl_{0,n}$ with compact support. 
A related applied approach for hypercomplex Clifford Fourier 
transformations\footnote{%
This is the kind of Clifford Fourier transform to which we will refer in
section \ref{sc:qftr}.} in $Cl_{0,n}$ was followed by B\"{u}low et. al.~\cite{GS:ncHCFT}.

By extending the classical trigonometric exponential function 
$\exp(j \, \mbox{\boldmath $x$} \ast \mbox{\boldmath $\xi $} )$ (where $\ast $ denotes the 
scalar product of $\mbox{\boldmath $x$}\in \R^m$ with $\mbox{\boldmath $\xi $}\in \R^m$,
$j$ the imaginary unit) 
in~\cite{LMQ:CAFT94,AM:CAFT96}, McIntosh et. al. generalized the classical
Fourier transform. Applied to a function of $m$ real variables this generalized
Fourier transform is holomorphic in $m$ complex variables and its inverse is 
\textit{monogenic}
in $m+1$ real variables, thereby effectively extending the function of $m$ real
variables to a monogenic function of $m+1$ real variables 
(with values in a \textit{complex} Clifford algebra). 
This generalization has significant applications to harmonic analysis,
especially to singular integrals on surfaces in $\R^{m+1}$. 
Based on this approach Kou and Qian obtained a Clifford Payley-Wigner theorem
and derived Shannon interpolation of band-limitted functions using the monogenic
sinc function~\cite[and references therein]{TQ:PWT}. 
The Clifford Payley-Wigner theorem also allows to derive left-entire 
(left-monogenic in the whole $\R^{m+1}$) functions from square integrable
functions on $\R^{m}$ with compact support. 

The real $n$-dimensional volume element
$i_n = \mbox{\boldmath $e$}_{1}\mbox{\boldmath $e$}_{2} \ldots \mbox{\boldmath $e$}_{n}$ 
of $Cl_{n,0}$ over the field
of the reals $\R$ has been used in \cite{ES:CFTonVF,HM:CFTUP,EM:CFaUP,HM:CFToMVF} 
to construct and apply Clifford Fourier 
transformations for $n=2,3 \, (\rm mod \, 4)$ with kernels 
$\exp(-i_n \vect{x}\ast \vect{u}),\; \vect{x}, \vect{u} \in \R^n$. 
This $i_n$ has a clear geometric interpretation. 
Note that $i_n^2=-1$ for $n=2,3 \, (\rm mod \, 4)$.

Ell \cite{TE:QFTpds} defined the quaternion Fourier transform (QFT) 
for application to 2D linear time-invariant systems of PDEs. 
Ell's QFT belongs to the growing family of Clifford Fourier transformations
because of (\ref{eq:isom}). But the left and right placement of the
exponential factors in definition \ref{df:QFT} distinguishes it. 
Later the QFT was applied extensively to 2D image processing,
including color images \cite{TB:thesis, TE:QFTpds, GS:ncHCFT}. 
This spurred research into optimized numerical 
implementations \cite{MF:thesis, PDC:effQFT}. 
Ell \cite{TE:QFTpds} and others \cite{TB:thesis, CCZ:comQuat} also
investigated related \textit{commutative} hypercomplex Fourier transforms 
like in the commutative subalgebra of $Cl_{4,0}$ with subalgebra basis 
$\{1, \vect{e}_{12}, \vect{e}_{34}, \vect{e}_{1234}\}$, 
\be
  \vect{e}_{12}^2=\vect{e}_{34}^2=-1, \quad \vect{e}_{1234}^2=+1\,\,.
\end{equation}

\begin{defn}[Quaternion Fourier transform (QFT)]
\label{df:QFT}
The quaternion Fourier transform\footnote{We also assume always that
$\int_{\R^2}|f(\vect{x})|\,d^2\vect{x}$ exists as well. But we do not
explicitly write this condition again in the rest of the paper.
} 
 $\hat{f}: \R^2 \rightarrow \HQ $ of 
$ f \in L^2(\R^2;\HQ)  $, 
$\vect{x}=x\vect{e}_1+y\vect{e}_2 \in \R^2 $, and
$\vect{u}=u\vect{e}_1+v\vect{e}_2 \in \R^2 $ is defined\footnote{%
For real signals $ f \in L^2(\R^2;\R)  $ 
the detailed relationship of the QFT of definition \ref{df:QFT} with the 
conventional scalar FT, i.e. with the even $\cos$-part and the odd 
$\sin$-part are given on pp. 191 and 192 of \cite{GS:ncHCFT}. 
With the help of (\ref{eq:QFTcp}) this can easily be extended 
to the full QFT of quaternion-valued $ f \in L^2(\R^2;\HQ)  $. 
} 
as
\be
  \hat{f}(\vect{u}) = \int_{\R^2} e^{-\sqi xu} f(\vect{x}) \,e^{-\sqj yv} d^2\vect{x}\, .
\label{eq:qft}
\end{equation}
\end{defn}
\noindent
The QFT can be \textit{inverted} by
\be
  f(\vect{x}) = \frac{1}{(2\pi)^2}\int_{\R^2} e^{\sqi xu} 
                \hat{f}(\vect{u}) \,e^{\sqj yv} d^2\vect{u} \, ,
  \label{eq:iqft}
\end{equation}
with $d^2\vect{u} = du dv$.

\subsection{Rewriting and splitting functions}

Let $f: \R^2 \rightarrow \HQ\,\,$ ( or $f \in L^2(\R^2;\HQ)$ ). 
Using four $\R^2 \rightarrow \R$ ( or $L^2(\R^2;\R)$ ) real component functions
$f_r, f_i, f_j,$ and $f_k$ we can decompose and rewrite $f$ with (\ref{eq:ijform}) as
\be
  f= f_r + f_i \qi + f_j \qj + f_k \qk = f_r + \qi f_i  + f_j \qj + \qi f_k \qj.
  \label{eq:flirj}
\end{equation}
We can also split the functions $f$ [similar to $q_{\pm}$ in (\ref{eq:pmform})] into
\be
  f = f_+ + f_-, \quad
  f_+ = \frac{1}{2}(f + \qi f \qj), \quad 
  f_- = \frac{1}{2}(f - \qi f \qj).
\label{eq:fpm}
\end{equation}
According to (\ref{eq:qpm}) the two components $f_{\pm}$ can also be rewritten as
\be
  f_{\pm} = \{f_r\pm f_k + \qi(f_i\mp f_j)\}\frac{1\pm \qk}{2}
       = \frac{1\pm \qk}{2} \{f_r\pm f_k + \qj(f_j\mp f_i)\}.
\label{eq:fpmcomp}
\end{equation} 
As an example let us consider the split of the product of exponential functions under the 
QFT integral in (\ref{eq:qft}). 
Using Euler's formula and trigonometric addition theorems 
the split leads to
\begin{gather}
  K = e^{-\sqi xu} e^{-\sqj yv} = K_+ + K_-, 
\nonumber \\
  K_{\pm} = e^{-\sqi(xu \mp yv)}\frac{1\pm\qk}{2} 
  = \frac{1\pm\qk}{2} e^{-\sqj(yv \mp xu)}.
  \label{eq:kpm}
\end{gather}

\subsection{Useful properties of the QFT \label{sc:UsefulPr}}

\begin{table}
\caption{Properties of the quaternion Fourier transform (QFT) 
\label{tb:QFTp1}
of quaternion functions
(Quat. Funct.) $f,g \in L^2(\R^2;\HQ)$, 
with $\vect{x}, \vect{u} \in \R^2$,
constants
$\alpha, \beta \in \{q | \,q=q_r+q_i  \qi, \, q_r,q_i\in\R \}$, 
$\alpha^{\,\prime}, \beta^{\prime} \in \{q | \,q=q_r+q_j  \qj, \, q_r,q_j\in\R \}$,
$a,b \in \R \setminus \{ 0\}$,  
$\vect{x}_0=x_0 \vect{e}_1+y_0 \vect{e}_2, \,
\vect{u}_0=u_0 \vect{e}_1+v_0 \vect{e}_2 \in \R^2$ and
$m,n \in \N_0$. }
\begin{center}
\begin{tabular}{llll} 
\hline
\textbf{Property}         &  \textbf{Quat. Funct.}      &   \textbf{QFT} 
\\ 
\hline
Left linearity      & $\alpha f(\vect{x})$+$\beta \;g(\vect{x})$     &  
$\alpha \hat{f}(\vect{u})$+ $\beta \hat{g}(\vect{u})$ 
\\ 
Right linearity     & $ f(\vect{x})\alpha^{\,\prime}$+$\;g(\vect{x})\beta^{\prime} $  &  
$\hat{f}(\vect{u})\alpha^{\,\prime}$+ $\hat{g}(\vect{u})\beta^{\prime}$ 
\\ 
$\vect{x}$-Shift    &  $f(\vect{x}-\vect{x}_0)$   & 
$e^{-\sqi x_0 u} \hat{f}(\vect{u}) \, e^{-\sqj y_0 v}$    
\\
Modulation    
  & $e^{\sqi x u_0} f(\vect{x}) \, e^{\sqj y v_0}$   
  & $\hat{f}(\vect{u}-\vect{u}_0)$ 
\\
Dilation\footnotemark  
  & $f(a \,x \vect{e}_1 + b \,y \vect{e}_2)$ 
  & $\frac{1}{|ab|}\hat{f}(\frac{u}{a}\vect{e}_1 + \frac{v}{b}\vect{e}_2)$
\\
Part. deriv. 
  & $ \frac{\partial^{m+n}}{\partial x^m \partial y^n}f(\vect{x}) $  
  & $ (\qi u)^m \hat{f}(\vect{u}) (\qj v)^n $ 
\\
Powers\footnotemark of $x,y$ 
  & $ x^m y^n f(\vect{x}) $ 
  & $ \qi^m \frac{\partial^{m+n}}{\partial u^m \partial v^n} \hat{f}(\vect{u}) \, \qj^n $
\\
Powers\addtocounter{footnote}{-1}\footnotemark 
  of $\qi,\qj$ 
  & $ \qi^m f(\vect{x}) \,\qj^n $ 
  & $ \qi^m \hat{f}(\vect{u}) \, \qj^n $
\\
Plancherel\addtocounter{footnote}{-1}\footnotemark 
  & $ \langle f,g \rangle =$ 
  & $ \frac{1}{(2 \pi)^2} \langle \hat{f},\hat{g} \rangle $ 
\\
Parseval\footnotemark 
  & $ \| f \| =$
  & $ \frac{1}{2\pi} \| \hat{f} \| $
\end{tabular}
\end{center}
\end{table}
\addtocounter{footnote}{-2}
\footnotetext{B\"{u}low \cite{TB:thesis} omits the absolute value 
  signs for the determinant of
  the transformation.}
\addtocounter{footnote}{1}\footnotetext{Theorems \ref{th:powxy}, \ref{th:powij} and  
  \ref{th:plan}.}
\addtocounter{footnote}{1}\footnotetext{Corollary  
  \ref{cr:pars}.}\addtocounter{footnote}{0}

We first show a 
\textit{new} Plancherel theorem with respect to the scalar product (\ref{eq:intsp}). 
\begin{thm}[QFT Plancherel]
\label{th:plan}
The scalar product (\ref{eq:intsp}) of two quaternion
module functions $ f,g \in L^2(\R^2;\HQ) $ is given by the scalar product of the of the
corresponding QFTs $\hat{f}$ and $\hat{g}$
\be
  \langle f,g \rangle = \frac{1}{(2 \pi)^2} \langle \hat{f},\hat{g} \rangle.
\label{eq:plan}  
\end{equation}
\end{thm}

\begin{proof}
For $ f,g \in L^2(\R^2;\HQ) $ we calculate the scalar product (\ref{eq:intsp})
\begin{gather}
  \langle f,g \rangle 
  = \int_{\R^2} \langle f(\vect{x})\tilde{g}(\vect{x}) \rangle_0 d^2\vect{x}
\nonumber \\  
  = \frac{1}{(2\pi)^2}\int_{\R^2} \langle 
    \int_{\R^2}e^{\sqi u x}\hat{f}(\vect{u})e^{\sqj v y} d^2\vect{u} \,
    \tilde{g}(\vect{x}) \rangle_0 d^2\vect{x}
\nonumber \\  
  = \frac{1}{(2\pi)^2}\int_{\R^2} \langle
    \hat{f}(\vect{u})
    \int_{\R^2}e^{\sqj v y} \tilde{g}(\vect{x})e^{\sqi u x}d^2\vect{x}
    \rangle_0 d^2\vect{u}
\nonumber \\  
  = \frac{1}{(2\pi)^2}\int_{\R^2} \langle
    \hat{f}(\vect{u})
    [\int_{\R^2} e^{-\sqi u x}g(\vect{x})e^{-\sqj v y}d^2\vect{x}]^{\sim}\,
    \rangle_0 d^2\vect{u}
\nonumber \\  
  = \frac{1}{(2\pi)^2}\int_{\R^2} \langle
    \hat{f}(\vect{u})
    \tilde{\hat{g}}{(\vect{u})}
    \rangle_0 d^2\vect{u}
  = \frac{1}{(2 \pi)^2} \langle \hat{f},\hat{g} \rangle. 
\label{eq:planp}
\end{gather}
In the second equality of (\ref{eq:planp}) we replaced $f$ with its inverse
QFT expression (\ref{eq:iqft}). In the third equality we exchanged the order of
integration and we used the cyclic symmetry (\ref{eq:cycs}). For the fourth equality
we simply pulled the reversion outside the square brackets $[\ldots]$ and obtained 
the QFT $\hat{g}(\vect{u})$, which proves (\ref{eq:plan}) according to (\ref{eq:intsp}).
\end{proof}

For $g=f$ the Plancherel theorem \ref{th:plan} has a QFT Parseval theorem (also called
Rayleigh's theorem) as a direct corollary. 
\begin{cor}[QFT Parseval]
\label{cr:pars}
The $L^2(\R^2;\HQ)$-norm of a quaternion module function $f\in L^2(\R^2;\HQ)$ is given
by the $L^2(\R^2;\HQ)$-norm of its QFT multiplied by $1/(2\pi)$
\be
  \|f\| = \frac{1}{2\pi}\|\hat{f}\|.
\label{eq:pars}
\end{equation}
\end{cor}

This leads to the following observations:
\begin{itemize}
  \item The way we obtained the Parseval theorem of cor.~\ref{cr:pars} is much simpler than 
the proofs in \cite{TB:thesis, TE:QFTpds}. 
  \item For two-dimensional linear time-invariant
partial differential systems the Parseval theorem provides an appropriate method
to measure controller performance. 
  \item In signal processing it states that the
signal energy is preserved by the QFT.
\end{itemize}

For solving PDEs with quaternionic (or real) coefficient polynomials in
$x,y \in \R^2$ we show the following two theorems.  
In this context we note again that every quaternionic (or real) 
coefficient polynomial in the
variables $x,y \in \R^2$ can be brought into a form having 
factors of $\qi\in \HQ$ to the left side of each term and 
factors of $\qj \in \HQ$ to the right side of each term (compare (\ref{eq:flirj})).  
\begin{thm}[Powers of $x,y$]
\label{th:powxy}
The QFT of a quaternion module function 
$x^my^nf(\vect{x})$ $\in L^2(\R^2;\HQ), \,
\vect{x}=x\vect{e}_1+y\vect{e}_2 \in \R^2, \, 
f\in L^2(\R^2;\HQ), \, m,n \in \N_0$
is given by
\be
  \widehat{x^my^n\hspace*{-0.7mm}f}(\vect{u}) 
  = \qi^m \frac{\partial^{m+n}}{\partial u^m \partial v^n} \hat{f}(\vect{u}) \, \qj^n.
\end{equation}
\end{thm}

\begin{proof}
The proof is done by induction. It is trivial for $m=n=0$. 
\\
For $m=1, n=0$ we calculate the QFT of $\widehat{xf}$ 
according to (\ref{eq:qft}) 
\begin{gather}
  \widehat{xf}(\vect{u}) 
  = \int_{\R^2} e^{-\sqi xu} xf(\vect{x}) \,e^{-\sqj yv} d^2\vect{x}
\nonumber \\
  = \int_{\R^2} \qi\frac{\partial}{\partial u}\,e^{-\sqi xu} 
    f(\vect{x}) \,e^{-\sqj yv} d^2\vect{x}
\nonumber \\
  = \qi\frac{\partial}{\partial u}
    \int_{\R^2} e^{-\sqi xu} f(\vect{x}) \,e^{-\sqj yv} d^2\vect{x}
  = \qi\frac{\partial}{\partial u} \hat{f}(\vect{u}).  
\label{eq:pow1}        
\end{gather}
In second equality we used 
$\frac{\partial}{\partial u}e^{-\sqi xu} = -\qi x e^{-\sqi xu}$ and $\qi (-\qi) = 1.$
\\
Completely analogous for $m=0,n=1$ we find
\be
\widehat{yf}(\vect{u}) 
  = \frac{\partial}{\partial v}
    \int_{\R^2} e^{-\sqi xu} f(\vect{x}) \,e^{-\sqj yv} d^2\vect{x} \, \qj
  = \frac{\partial}{\partial v} \hat{f}(\vect{u}) \qj \, . 
\end{equation}
Because of non-commutativity $\qj$ appears to the right of $\hat{f}$. 
Induction over
$m,n \in \N$ completes the proof.
\end{proof}

\begin{thm}[Powers of $\qi, \qj$]
\label{th:powij}
The QFT of a quaternion module function 
$ \qi^m f(\vect{x}) \qj^n$ $\in L^2(\R^2;\HQ), \,
 \, 
f \in L^2(\R^2;\HQ), \, m,n \in \N_0$
is given by
\be
  \widehat{\qi^m f \qj^n}(\vect{u}) 
  = \qi^m \hat{f}(\vect{u}) \, \qj^n.
\end{equation}
\end{thm}

\begin{proof}
  Similar to the left and right linearities of table \ref{tb:QFTp1} 
  theorem \ref{th:powij} follows directly from the definition \ref{df:QFT} 
  of the QFT, using the commutation relationships 
\be
  \exp(-\qi xu) \qi^m = \qi^m\exp(-\qi xu) \quad
  \text{and} \quad
  \exp(-\qj yv) \qj^n = \qj^n\exp(-\qj yv).
\end{equation}
\end{proof}

For every $f \in L^2(\R^2;\HQ)$ we can always rewrite 
$f=f_r + f_i \qi + f_j \qj + f_k \qk$ 
as in (\ref{eq:flirj}) to the form
\be
  f = f_r + \qi f_i  + f_j \qj + \qi f_k \qj.
\end{equation}
Accordingly we now can make the following two important observations:
\begin{itemize}
\item 
Theorem \ref{th:powij} reduces the computation
of the QFT of any $f \in L^2(\R^2;\HQ)$ to the computation of four QFTs of the 
real functions 
$f_r, f_i, f_j, f_k \in L^2(\R^2;\R)$ as in
\be
  \hat{f} = \hat{f}_r + \qi \hat{f}_i  + \hat{f}_j \qj + \qi \hat{f}_k \qj.
\label{eq:QFTcp}
\end{equation}
 \item On the other hand theorem \ref{th:powij} reveals that every 
\textit{theorem for the 
QFT of real functions} $g \in L^2(\R^2;\R)$ immediately results via (\ref{eq:QFTcp})
in a corresponding \textit{theorem for quaternion module functions} 
$f \in L^2(\R^2;\HQ)$. We simply need to apply the theorem for the QFT of real functions
to each of the four real component functions $f_r, f_i, f_j, f_k\in L^2(\R^2;\R)$. 
This fact is rather useful, because often in image processing theorems are only
established for real image signals \cite{TB:thesis}. 
\end{itemize}

\subsection{Example: $GL(\R^2)$ transformation properties of the QFT \label{sc:exGL}}

To give an example for the second observation at the end of section \ref{sc:UsefulPr}
we use it to generalize the
general linear real non-singular transformation property of the QFT of 
real 2D functions $f \in L^2(\R^2;\R)$ of \cite{TB:thesis} to quaternion module
functions $f \in L^2(\R^2;\HQ)$. 
\\
This property of real 2D signals states that for
\be
  \vect{x}^{\prime}= {\mathcal A}
                     \vect{x} = (ax+by)\vect{e}_1+(cx+dy)\vect{e}_2
\label{eq:GLtrafo}
\end{equation}
with non-singular real transformation matrix
\be
  A = \left(
      \begin{array}{cc}
      a & b \\
      c & d
      \end{array}
      \right)
  \label{eq:Amatrix}
\end{equation}
the QFT of a \textit{real} signal $f: \R^2 \rightarrow \R$  
is\footnote{B\"{u}low \cite{TB:thesis} omits the absolute value signs for the determinant of
the transformation.}\addtocounter{footnote}{-1}
\be
  \widehat{f({\mathcal A}\vect{x})}(\vect{u})
  = \frac{| \det {\mathcal B}|}{2} \left( 
    \hat{f}({\mathcal B}_+\,\vect{u})+\hat{f}({\mathcal B}_-\,\vect{u})
    + \qi \left\{ \hat{f}({\mathcal B}_+\, \vect{u})-\hat{f}({\mathcal B}_-\, \vect{u})\right\} \qj
                      \right).
\label{eq:Atrafo}
\end{equation}
In (\ref{eq:Atrafo}) the two linear real non-singular transformations 
${\mathcal B}_+$ and ${\mathcal B}_-$ have corresponding matrices and the 
(same) determinant
\begin{gather}
  B_+ = {A^{-1}}^{T}, \quad 
  B_- = \frac{1}{\det A}
        \left(
        \begin{array}{cc}
        d & c \\
        b & a
        \end{array}
        \right), 
\nonumber \\
  \det {\mathcal B} = \det B_+ = \det B_- = (\det A )^{-1} .     
\label{eq:Btrafos}
\end{gather} 
We can now establish the generalization from $f \in L^2(\R^2;\R)$
to $f \in L^2(\R^2;\HQ)$ functions. 
\begin{thm}
\label{th:GLqft}
The QFT of a quaternion-module function $f \in L^2(\R^2;\HQ)$ with
a $GL(\R^2)$ transformation $\mathcal{A}$ of its vector arguments
(\ref{eq:GLtrafo}) is also given by (\ref{eq:Atrafo}).
\end{thm}
\begin{proof}
We only sketch the proof, because writing out all expressions explicitly
would consume too much space:
\begin{itemize}
  \item Applying (\ref{eq:Atrafo}) and (\ref{eq:Btrafos}) to each component of 
(\ref{eq:QFTcp}) and 
  \item rearranging the sum (of 16 terms) yields the validity of 
(\ref{eq:Atrafo}) together with (\ref{eq:Btrafos}) also for quaternion-valued 
$f \in L^2(\R^2;\HQ)$. 
  \item It is again crucial that in each 
term all factors $\qi$ are always kept to the left and all factors $\qj$ 
are always kept to the right.
\end{itemize}
\end{proof}

We remark that resorting to matrices and matrix manipulations is geometrically 
not very intuitive, so in section \ref{sc:GLprop} an alternative more geometric 
approach is taken to derive the transformation properties of general
$f \in L^2(\R^2;\HQ)$. This geometric approach has far reaching consequences
for the generalization of the QFT, exploited in later sections. 

But before geometrically reanalyzing QFT transformation properties we look
at the following variant of the QFT with some desirable properties not valid 
for the QFT of definition \ref{df:QFT}.

\section{The right side quaternion Fourier transform (QFTr) \label{sc:qftr}}

We observe that it is not possible to establish a general Plancherel theorem 
for the QFT of the inner product $(f,g)$ of (\ref{eq:intip}), 
because the product (\ref{eq:intip}) lacks the cyclic
symmetry (\ref{eq:cycs}) applied in the proof of theorem \ref{th:plan}. 
To obtain a 
Plancherel theorem it is therefore either necessary to modify the symmetry 
properties of 
the inner product as in (\ref{eq:intsp}) or to modify the QFT itself. In this
section we explore the second possibility.
\begin{defn}[Right side QFT (QFTr)]
The right side quaternion Fourier transform $\qftr{f}: \R^2 \rightarrow \HQ $ of 
$ f \in L^2(\R^2;\HQ) $, 
$\vect{x}=x\vect{e}_1+y\vect{e}_2 \in \R^2 $, and
$\vect{u}=u\vect{e}_1+v\vect{e}_2 \in \R^2 $ is defined as
\be
  \qftr{f}(\vect{u}) = \int_{\R^2} f(\vect{x}) \,e^{-\sqi xu} e^{-\sqj yv} d^2\vect{x}
  \quad \quad
  \text{with} \quad d^2\vect{x} = dx dy. 
\label{eq:qftr}
\end{equation}
\end{defn}
The QFTr is known as Clifford Fourier transform~\cite{BDS:CA, GS:ncHCFT}, 
because of the isomorphism $\HQ \cong Cl_{0,2}$. Further freedoms in alternative 
definitions would be to exchange the order of the exponentials in (\ref{eq:qftr})
or to wholly shift both exponential factors to the left side instead. 
The former
would simply exchange the roles of $\qi$ and $\qj$, but the latter would not serve
our purpose as will soon become clear.
The QFTr can be inverted~\cite{BDS:CA, GS:ncHCFT} using
\be
  f(\vect{x}) 
  = \frac{1}{(2\pi)^2}\int_{\R^2} \qftr{f}(\vect{u}) 
    \,e^{\sqj yv}e^{\sqi xu} d^2\vect{u} ,
\label{eq:iqftr}
\end{equation}
with $d^2\vect{u} = du dv$. Attention needs to be paid to the reversed 
order of the exponential factors in (\ref{eq:iqftr}) compared to (\ref{eq:qftr}).

\subsection{Properties of the QFTr}

For general $f,g \in L^2(\R^2;\HQ)$ \textit{left linearity}
and \textit{dilation} properties of table \ref{tb:QFTp1} hold. The
left linearity coefficients can now be fully quaternionic constants
$\alpha^{\prime}, \beta^{\prime} \in \HQ$.

\begin{table}
\caption{Properties of the right sided quaternion Fourier transform (QFTr) 
\label{tb:QFTr}
of quaternion functions
(Quat. Funct.) $f,g \in L^2(\R^2;\HQ)$, 
with $\vect{x}, \vect{u} \in \R^2$,
constants
$\alpha, \beta \in \R$, 
$\alpha^{\,\prime}, \beta^{\prime} \in \HQ$,
$a,b \in \R\setminus\{0\}$,  
$\vect{x}_0=x_0 \vect{e}_1+y_0 \vect{e}_2, 
\vect{u}_0=u_0 \vect{e}_1+v_0 \vect{e}_2 \in \R^2$ and
$m,n \in \N$. }
\begin{center}
\begin{tabular}{llll} 
\hline
\textbf{Property}         &  \textbf{Quat. Funct.}      &   \textbf{QFTr} 
\\ 
\hline
Linearity\addtocounter{footnote}{1}\footnotemark      & $\alpha f(\vect{x})$+$\beta \;g(\vect{x})$     &  
$\alpha \qftr{f}(\vect{u})$+ $\beta \qftr{g}(\vect{u})$ 
\\ 
Left linearity     & $ \alpha^{\,\prime}f(\vect{x})$+$\;\beta^{\prime}g(\vect{x}) $  &  
$\alpha^{\,\prime}\qftr{f}(\vect{u})$+ $\beta^{\prime}\qftr{g}(\vect{u})$ 
\\ 
$\vect{x}$-Shift\footnotemark    
  &  $f(\vect{x}-\vect{x}_0)$   
  &  $\mathcal{F}_{\triangleright} \{f e^{-\sqi x_0 u}\} (\vect{u}) \, e^{-\sqj y_0 v}$  
\\
Dilation
  & $f(a \,x \vect{e}_1 + b \,y \vect{e}_2)$ 
  & $\frac{1}{|ab|}\qftr{f}(\frac{u}{a}\vect{e}_1 + \frac{v}{b}\vect{e}_2)$
\\
Part. deriv.\footnotemark 
  & $ \frac{\partial^{m+n}}{\partial x^m \partial y^n}f(\vect{x}) \qi^{-m} $  
  & $ u^m \qftr{f}(\vect{u}) (\qj v)^n  $ 
\\
Powers\addtocounter{footnote}{0}\footnotemark of $x,y$ 
  & $ x^m y^n f(\vect{x}) \qi^{-m} $ 
  & $ \frac{\partial^{m+n}}{\partial u^m \partial v^n} \qftr{f}(\vect{u}) \, \qj^n $
\\
Powers\addtocounter{footnote}{0}\footnotemark
  of $\qi,\qj$ 
  & $ \qi^m \qj^n f(\vect{x}) $ 
  & $ \qi^m \qj^n \qftr{f}(\vect{u})  $
\\
Plancherel\footnotemark 
  & $ ( f,g ) =$ 
  & $ \frac{1}{(2 \pi)^2} ( \qftr{f},\qftr{g} ) $ 
  \\
Plancherel\footnotemark 
  & $ \langle f,g \rangle =$ 
  & $ \frac{1}{(2 \pi)^2} \langle \qftr{f},\qftr{g} \rangle $ 
\\
Parseval
  & $ \| f \| =$
  & $ \frac{1}{2\pi} \| \qftr{f} \| $
\end{tabular}
\end{center}
\end{table}
\addtocounter{footnote}{-6}
\footnotetext{The positions of the real scalars $\alpha, \beta$ before or
after the functions $f, g$ do not matter.}
\addtocounter{footnote}{1}
\footnotetext{%
Only for
quaternion module functions $f\in L^2(\R^2;\HQ)$ with $\qi f = f \qi$, 
i.e. $f=f_r+\qi f_i  \text{ with } f_r,f_i \in L^2(\R^2;\R)$ do 
we get
$\mathcal{F}_{\triangleright}\{f(\vect{x}-\vect{x}_0)\}(\vect{u})  
= e^{-\sqi x_0 u} \qftr{f}(\vect{u}) \, e^{-\sqj y_0 v}$.  
}
\addtocounter{footnote}{1}
\footnotetext{%
Only for $\qi f = f \qi$ do we get
$ \mathcal{F}_{\triangleright}\{\frac{\partial^{m+n}}{\partial x^m \partial y^n}f\}(\vect{u})
= (\qi u)^m \qftr{f}(\vect{u}) (\qj v)^n $.
}
\addtocounter{footnote}{1}
\footnotetext{%
Only for $\qi f = f \qi$ do we get
$ \mathcal{F}_{\triangleright}\{x^m y^n f\}(\vect{u}) 
= \qi^m \frac{\partial^{m+n}}{\partial u^m \partial v^n} \qftr{f}(\vect{u}) \, \qj^n $.
}
\addtocounter{footnote}{1}
\footnotetext{Here the powers of $\qi$, $\qj$ law
is a direct consequence of the left linearity.}
\addtocounter{footnote}{1}
\footnotetext{Compare theorem \ref{th:qplan}.}
\addtocounter{footnote}{1}
\footnotetext{A direct consequence of symmetrizing theorem \ref{th:qplan}.}

But $\vect{x}$-shift, partial derivative,
and powers of $x^m y^n$ properties 
need to be modified as in table \ref{tb:QFTr}.
Regarding (\ref{eq:quat}) it is clear that 
$\qi f = f \qi$ holds iff $f=f_r + f_i \,\qi, \,\, f_r, f_i\in \R$, which is slightly more
general than the restriction of \cite{TB:thesis} to $f=f_r \in \R$. 
A modulation property analogous to the one in table \ref{tb:QFTp1} does not
 hold. It is obstructed by the non-commutativity of the exponential factors
\be
  \exp(\qj yv_0) \, \exp(\qi xu) \neq \exp(\qi xu) \, \exp(\qj yv_0).
\end{equation}
For a powers of $\qi,\qj$ property to hold for the QFTr, we need to 
shift the factors $\qj^n$ also to the left of the quaternion function $f(\vect{x})$.

For fully general quaternion-valued $f,g \in L^2(\R^2;\HQ)$ we can establish for the QFTr the following
quaternion-valued Plancherel theorem based on the inner product (\ref{eq:intip}).

\begin{thm}[QFTr Plancherel]
\label{th:qplan}
The (quaternion-valued) inner product (\ref{eq:intip}) 
of two quaternion module functions $ f,g \in L^2(\R^2;\HQ) $ is given by the 
inner product of the corresponding QFTrs $\qftr{f}$ and $\qftr{g}$
\be
  ( f,g ) = \frac{1}{(2 \pi)^2} ( \qftr{f},\qftr{g} ).
\label{eq:qplan}  
\end{equation}
\end{thm}

\begin{proof}
For $ f,g \in L^2(\R^2;\HQ) $ we calculate the inner product (\ref{eq:intip})
\begin{gather}
  ( f,g ) 
  = \int_{\R^2} f(\vect{x})\tilde{g}(\vect{x}) d^2\vect{x}
\nonumber \\  
  = \frac{1}{(2\pi)^2}\int_{\R^2}  
    \int_{\R^2}\qftr{f}(\vect{u})e^{\sqj v y} e^{\sqi u x} d^2\vect{u} \,
    \tilde{g}(\vect{x})  d^2\vect{x}
\nonumber \\  
  = \frac{1}{(2\pi)^2}\int_{\R^2} 
    \qftr{f}(\vect{u})
    \int_{\R^2}e^{\sqj v y} e^{\sqi u x} \tilde{g}(\vect{x})d^2\vect{x}
     d^2\vect{u}
\nonumber \\  
  = \frac{1}{(2\pi)^2}\int_{\R^2} 
    \qftr{f}(\vect{u})
    [\int_{\R^2} g(\vect{x})e^{-\sqi u x}e^{-\sqj v y}d^2\vect{x}]^{\sim}\,
     d^2\vect{u}
\nonumber \\  
  = \frac{1}{(2\pi)^2}\int_{\R^2} 
    \qftr{f}(\vect{u})
    \tilde{\qftr{g}}{(\vect{u})}
     d^2\vect{u}
  = \frac{1}{(2 \pi)^2} ( \qftr{f},\qftr{g} ). 
\label{eq:qplanp}
\end{gather}
In the second equality of (\ref{eq:qplanp}) we replaced $f$ with its inverse
QFTr expression (\ref{eq:iqftr}). In the third equality we exchanged the order of
integration. For the fourth equality we simply pulled the reversion outside the 
square brackets $[\ldots]$ and obtained the QFTr $\qftr{g}(\vect{u})$, 
which proves (\ref{eq:qplan}) according to (\ref{eq:intip}).
\end{proof}
For $g=f$ theorem \ref{th:qplan} has a corresponding QFTr Parseval theorem as a 
direct corollary. 
\begin{cor}[QFTr Parseval]
\label{cr:qpars}
The $L^2(\R^2;\HQ)$-norm of a quaternion module function $f\in L^2(\R^2;\HQ)$ is given
by the $L^2(\R^2;\HQ)$-norm of its QFTr $\qftr{f}$ multiplied by $1/(2\pi)$
\be
  \|f\| = \frac{1}{2\pi}\|\qftr{f}\|=\frac{1}{2\pi}\|\hat{f}\|.
\label{eq:qpars}
\end{equation}
\end{cor}

\begin{proof}
The first identity follows from setting $g=f$ in theorem \ref{th:qplan} (QFTr Plancherel). The second
identity follows from comparing with corollary \ref{cr:pars} (QFT Parseval).
\end{proof}

To facilitate the use of the QFTr and comparison with the QFT (table \ref{tb:QFTp1}) 
we list the main QFTr properties in table \ref{tb:QFTr}.

\section{Understanding the $GL(\R^2)$ transformation 
         properties of the QFT \label{sc:GLprop}}

We begin with noting that the matrix transformation law (\ref{eq:Atrafo}),
derived by B\"{u}low \cite{TB:thesis} for real signals
$f \in L^2(\R^2;\R)$, and generalized in theorem \ref{th:GLqft}
of section \ref{sc:exGL}
  to quaternion-valued signals\footnote{%
Remember that B\"{u}low \cite{TB:thesis} proved his transformation law only
for real signals. But in theorem \ref{th:GLqft} of section \ref{sc:exGL},  
we used (\ref{eq:QFTcp}) and theorem \ref{th:powij}
to generalize from real signals $f\in L^2(\R^2,\R)$ to quaternion
valued signals $f\in L^2(\R^2,\HQ)$.
} 
$f \in L^2(\R^2;\HQ)$, with four terms on the right side, allows no 
straightforward
geometric interpretation. Yet a clear geometric interpretation is not only needed
in many applications, such an interpretation is also very instructive in order to successfully 
generalize the QFT to higher dimensions. 

Toward this aim we observe, that the split (\ref{eq:kpm}) of the exponentials $K$ under
the QFT integral
results in two (single exponential) complex kernels $K_{\pm}$ with complex 
units $\qi$ (or $\qj$) apart from the right (or left) factor $(1\pm \qk)/2$. 

This and the known elegant monomial transformation
properties of complex Fourier transforms (also preserved in the 
Clifford FT of \cite{HM:CFTUP})
motivates us to geometrically re-analyze the $GL(\R^2)$ transformation properties 
of the QFT of $f \in L^2(\R^2;\HQ)$ in terms of its two components $f_{\pm}$
as given in (\ref{eq:fpm}).

\begin{thm}[QFT of $f_{\pm}$]
\label{th:fpmtrafo}
The QFT of the $f_{\pm}$ split parts of a quaternion module function 
$f \in L^2(\R^2,\HQ)$ have the complex forms
\be
 \hat{f}_{\pm} 
  \stackrel{}{=} \int_{\R^2}
    f_{\pm}e^{-\sqj (yv \mp xu)}d^2x
  \stackrel{}{=} \int_{\R^2}
    e^{-\sqi (xu \mp yv)}f_{\pm}d^2x \,\, .
\end{equation}
\end{thm}
\begin{proof}
\begin{gather}
  \hat{f}_{\pm} 
  = \int_{\R^2}e^{-\sqi xu}
    \{f_r\pm f_k + \qi(f_i\mp f_j)\}\frac{1\pm \qk}{2}e^{-\sqj yv}d^2x
\nonumber \\
  = \int_{\R^2}
    \{f_r\pm f_k + \qi(f_i\mp f_j)\}\,e^{-\sqi xu}\frac{1\pm \qk}{2}e^{-\sqj yv}d^2x
\nonumber \\
  \stackrel{}{=} \int_{\R^2}
    \{f_r\pm f_k + \qi(f_i\mp f_j)\} \,
    \underbrace{\frac{1\pm \qk}{2}e^{-\sqj (yv \mp xu)}}_{=K_{\pm}}\,\,d^2x
\nonumber \\
  \stackrel{(\ref{eq:fpmcomp})}{=} \int_{\R^2}
    f_{\pm}e^{-\sqj (yv \mp xu)}d^2x
  \stackrel{}{=} \int_{\R^2}
    e^{-\sqi (xu \mp yv)}f_{\pm}d^2x \,\, ,
  \label{eq:QFTfpm}
\end{gather}
where for the third equality we did a number of quaternion algebra manipulations,
involving Euler's formula and trigonometric addition theorems. 
The last equality of (\ref{eq:QFTfpm}) follows analogously by replacing
$f_{\pm}$ with the third expression in (\ref{eq:fpmcomp}), etc.
\end{proof}
We learn from the third line of (\ref{eq:QFTfpm}) that the behavior of 
the two parts (\ref{eq:kpm}) under automorphisms ${\mathcal A}\in GL(\R^2)$ 
also determines the automorphism properties of the QFTs $\hat{f}_{\pm}$, where 
due to theorem \ref{th:powij} the QFT operation and the split operation 
(\ref{eq:fpm}) commute.

\subsection{Geometric interpretation and coordinate independent formulation
            of $GL(\R^2)$ transformations of the QFT}

We begin with noting that according to the \textit{polar decomposition theorem} 
\cite{DH:NF1} 
every automorphism ${\mathcal A}\in GL(\R^2)$ has a unique decomposition 
${\mathcal A}={\mathcal T}{\mathcal R} ={\mathcal R}{\mathcal S}$, where ${\mathcal R}$ is a rotation and
${\mathcal T}$ and ${\mathcal S}$ are symmetric with positive and negative eigenvalues. 

Positive eigenvalues correspond to stretches by the eigenvalue 
in the direction of the eigenvector. Negative eigenvalues correspond to 
reflections at the line (hyperplane) normal to the eigenvector, 
composed with stretches by the absolute value of the eigenvalue
 in the direction of the eigenvector.

Stretches (positive eigenvalues) ${\mathcal D} \in GL(\R^2)$ were already fully 
treated in \cite{TB:thesis} (compare also table \ref{tb:QFTp1}).

Rotations correspond to two reflections \cite{CM:DiscGr, EC:GrSimp}
at lines subtending half the 
angle of the resulting rotation 
${\mathcal R}_{\vect{a}\vect{b}}={\mathcal U}_{\vect{a}}{\mathcal U}_{\vect{b}}$. 
The elementary transformations that compose all
automorphisms ${\mathcal A}\in GL(\R^2)$ are therefore stretches and reflections.

In geometric algebra reflections ${\mathcal U}_{\vect{n}}$ at a hyperplane 
(line in 2D) through the origin can be characterized 
by normal vectors $\vect{n}$
\be
  {\mathcal U}_{\vect{n}}\vect{x} = -\vect{n}^{-1}\vect{x}\vect{n}. 
  \label{eq:garef}
\end{equation}
The length of $\vect{n}$ does not matter. ${\mathcal U}_{\vect{n}}$ preserves (reverses)
the component parallel (perpendicular) to the hyperplane of reflection.

With the vectors $\vect{x}=x\vect{e}_1+y\vect{e}_2$ , $\vect{u}=u\vect{e}_1+v\vect{e}_2$
we now rewrite coordinate free\footnote{%
The fact that the reflection ${\mathcal U}_{\vect{e}_1}$
with the special hyperplane normal to vector $\vect{e}_1$
is needed stems from the arbitrary initial association
of the $\vect{e}_1$-coordinate product $xu$ with $\qi$ 
and of the $\vect{e}_2$-coordinate product $yv$ with $\qj$. 
}
 the \textit{angles in the exponentials} 
of $\hat{f}_{\pm}$ as
\be
  -xu+yv = \vect{x} \cdot ({\mathcal U}_{\vect{e}_1}\vect{u}), \quad 
  xu+yv = \vect{x} \cdot \vect{u}.
\end{equation}
Hence we get for the QFTs of $f_\pm$
\be
 \hat{f}_{+} 
  \stackrel{}{=} \int_{\R^2}
    f_{+}e^{-\sqj \, \vect{x} \cdot ({\mathcal U}_{\vect{e}_1}\vect{u})}d^2x,
    \quad
  \hat{f}_{-} 
  \stackrel{}{=} \int_{\R^2}
    f_{-}e^{-\sqj \, \vect{x} \cdot \vect{u}}d^2x.
\end{equation}
The QFT of $f_-$ is analogous to a complex 2D Fourier transform,
only in general $f_-$ and the exponential factor do not commute. 
The QFT of $f_+$ is similar except for the reflection ${\mathcal U}_{\vect{e}_1}$.

We are now in a position to apply any automorphism ${\mathcal A}\in GL(\R^2)$ to the spatial
argument of the $f_{\pm}$ components of any $f \in L^2(\R^2,\HQ)$. We begin with
\begin{gather}
{ \widehat{f_-({\mathcal A}\vect{x})}(\vect{u}) }
  = \int_{\R^2} f_-({\mathcal A}\vect{x}) e^{-\sqj \, \vect{x} \cdot \vect{u}}d^2x
\nonumber \\
  \stackrel{\vect{z} = {\mathcal A}\vect{x}}{=}
  \int_{\R^2} f_-(\vect{z}) 
  e^{-\sqj ({\mathcal A}^{-1}\vect{z}) \cdot \vect{u}}\,|\det {\mathcal A}^{-1}|d^2z
\nonumber \\
  =  |\det {\mathcal A}^{-1}| \int_{\R^2} f_-(\vect{z}) 
  e^{-\sqj \, \vect{z} \cdot (\overline{{\mathcal A}^{-1}}\vect{u})}\,
  d^2z
\nonumber \\
  =  
|\det {\mathcal A}^{-1}| \, \hat{f}_-(\overline{{\mathcal A}^{-1}}\vect{u}),
\label{eq:fmtrafo}
\end{gather}
where $\overline{{\mathcal A}^{-1}}$ indicates the adjoint automorphism of
${\mathcal A}^{-1}$. The absolute value of the determinant $\det {\mathcal A}^{-1}$ needs
to be used, because of the interchange of integration boundaries for a negative 
determinant. 
We continue with
\begin{gather}
\widehat{f_+({\mathcal A}\vect{x})}(\vect{u})
  = \int_{\R^2} f_+({\mathcal A}\vect{x}) 
    e^{-\sqj \, \vect{x} \cdot ({\mathcal U}_{\vect{e}_1}\vect{u})}d^2x
\nonumber \\
  \stackrel{\vect{z} = {\mathcal A}\vect{x}}{=}
     \int_{\R^2} f_+(\vect{z}) 
     e^{-\sqj ({\mathcal A}^{-1}\vect{z}) \cdot ({\mathcal U}_{\vect{e}_1}\vect{u})}\,
     |\det {\mathcal A}^{-1}|d^2z
\nonumber \\
  = |\det {\mathcal A}^{-1}| \int_{\R^2} f_+(\vect{z}) 
    e^{-\sqj \, \vect{z} \cdot (\overline{{\mathcal A}^{-1}}{\mathcal U}_{\vect{e}_1}\vect{u})}\,
    d^2z
\nonumber \\
  = |\det {\mathcal A}^{-1}| \int_{\R^2} f_+(\vect{z}) 
    e^{-\sqj \, \vect{z} \cdot 
    ({\mathcal U}_{\vect{e}_1}{\mathcal U}_{\vect{e}_1}\overline{{\mathcal A}^{-1}}
    {\mathcal U}_{\vect{e}_1}\vect{u})}\,
    d^2z
\nonumber \\
  = 
|\det {\mathcal A}^{-1}| \, \hat{f}_+
    ({\mathcal U}_{\vect{e}_1}\overline{{\mathcal A}^{-1}}\,{\mathcal U}_{\vect{e}_1}\vect{u}),
\label{eq:fptrafo}
\end{gather}
which is very similar to the previous calculation for $\hat{f}_-$. 
The only difference is that in line 4 we insert 
$1={\mathcal U}_{\vect{e}_1}{\mathcal U}_{\vect{e}_1}$ before $\overline{{\mathcal A}^{-1}}$,
and that the argument of the transformed $\hat{f}_+$ now has the \textit{reflected} version 
${\mathcal U}_{\vect{e}_1}\overline{{\mathcal A}^{-1}}\,{\mathcal U}_{\vect{e}_1}$
of the adjoint inverse transformation $\overline{{\mathcal A}^{-1}}$. 
Recombining $\hat{f}_+$ and $\hat{f}_-$ we get from (\ref{eq:fmtrafo}) and (\ref{eq:fptrafo})
\begin{thm}[$GL(\R^2)$ transformation properties of the QFT]
\label{th:GLtrafo}
  The QFT of a quaternion module function $f \in L^2(\R^2;\HQ)$ with a 
  $GL(\R^2)$ transformation ${\mathcal A}$ of its vector argument is given by 
  \begin{gather}
    \widehat{{f}({\mathcal A}\vect{x})}(\vect{u}) 
    = |\det {\mathcal A}^{-1}|\, \{\, \hat{f}_-(\overline{{\mathcal A}^{-1}}\vect{u})
      + \, \hat{f}_+
        ({\mathcal U}_{\vect{e}_1}\overline{{\mathcal A}^{-1}}\,
         {\mathcal U}_{\vect{e}_1}\vect{u}) \,\} \,\, .       
  \end{gather}
\end{thm}
Theorem \ref{th:GLtrafo} corresponds exactly to equation
(\ref{eq:Atrafo}) with (\ref{eq:Btrafos}), if the matrix expression 
(\ref{eq:Amatrix}) is used for 
the automorphism ${\mathcal A}$ and if the $f_{\pm}$ split formulas 
(\ref{eq:fpm}) are used. The four terms of (\ref{eq:Atrafo}) together with
all the matrices involved therefore get in theorem \ref{th:GLtrafo} a clear
geometric interpretation. In order to be even more explicit we specify below the
full geometric algebra expressions for stretches, reflections and rotations.

\subsection{Explicit examples: stretches, reflections \& rotations}

To deepen our geometrical understanding we now look at stretches, reflections (and
rotations) which \textit{compose} every general automorphism ${\mathcal A} \in GL(\R^2)$.

Stretches expressed by
${\mathcal A}_s \vect{x}= a x \vect{e}_1+ b y \vect{e}_2, 
\text{ with } a,b\in \R\setminus \{0\}$,
result because of 
$\,\,{\mathcal U}_{\vect{e}_1}{\mathcal A}_s \, {\mathcal U}_{\vect{e}_1}={\mathcal A}_s$ in
\be
  \widehat{f({\mathcal A}_s \vect{x})}(\vect{u}) 
  = |\det {\mathcal A}_s^{-1}|\, \hat{f}({\mathcal A}_s^{-1}\vect{u})
  = \frac{1}{|ab|}\hat{f}(\frac{u}{a}\vect{e}_1+\frac{v}{b}\vect{e}_2).
\label{eq:dil}
\end{equation}

Reflections in hyperplanes normal to 
  $\vect{a}$ expressed by
  ${\mathcal U}_{\vect{a}}\vect{x} = -\vect{a}^{-1}\vect{x}\vect{a}$ ,
  with
  $\,\,|\det {\mathcal U}_{\vect{a}}| = 1\, ,$
  $\,\,\overline{{\mathcal U}_{\vect{a}}} ={\mathcal U}_{\vect{a}}$ ,
  $\,\, {\mathcal U}_{\vect{e}_1}{{\mathcal U}_{\vect{a}}}\,{\mathcal U}_{\vect{e}_1} = 
   {\mathcal U}_{\vect{a}^{\prime}}\, ,$ 
  and $\vect{a}^{\prime}={\mathcal U}_{\vect{e}_1}\vect{a}$
  result in
  \be
  \widehat{f({\mathcal U}_{\vect{a}}\vect{x})}(\vect{u}) 
  = \hat{f}_-({{\mathcal U}_{\vect{a}}}\vect{u})
      + \, \hat{f}_+
    ({\mathcal U}_{\vect{a}^{\prime}}\vect{u}).
  \end{equation}

Finally rotations (equivalent to two reflections at lines subtending 
half the rotation angle) expressed by
  $\,\,{\mathcal R}_{\vect{a}\vect{b}}\vect{x} 
   = {\mathcal U}_{\vect{b}}{\mathcal U}_{\vect{a}}\vect{x}$ ,
  with
  $\,|\det {\mathcal R}_{\vect{a}\vect{b}}| = 1\, ,$  
  ${\mathcal R}^{-1}_{\vect{a}\vect{b}} = {\mathcal R}_{\vect{b}\vect{a}}\, ,$ 
  and
  ${\mathcal U}_{\vect{a}^{\prime}}{\mathcal U}_{\vect{b}^{\prime}} 
   = {\mathcal U}_{\vect{e}_1} {\mathcal R}^{-1}_{\vect{a}\vect{b}}{\mathcal U}_{\vect{e}_1}\, ,$ 
  result in
  \begin{gather}
    \widehat{f({\mathcal R}_{\vect{a}\vect{b}}\vect{x})}(\vect{u}) 
      =\hat{f}({\mathcal U}_{\vect{b}}{\mathcal U}_{\vect{a}}\vect{x})(\vect{u}) 
      = \hat{f}_-({{\mathcal U}_{\vect{a}}{\mathcal U}_{\vect{b}}}\vect{u})
        + \, \hat{f}_+
        ({\mathcal U}_{\vect{a}^{\prime}}{\mathcal U}_{\vect{b}^{\prime}}\vect{u})
    \nonumber \\
      = \hat{f}_-({\mathcal R}^{-1}_{\vect{a}\vect{b}}\vect{u})
        + \, \hat{f}_+
        ({\mathcal U}_{\vect{e}_1} {\mathcal R}^{-1}_{\vect{a}\vect{b}}\,{\mathcal U}_{\vect{e}_1}\vect{u})
  \label{eq:frot} 
  \end{gather}
In two dimensions\footnote{%
In section \ref{sc:gen} we generalize theorem \ref{th:GLtrafo} to higher dimensions,
but for rotations the expression for $\hat{f}_+$ on the right hand side of (\ref{eq:frot2D})
will in general not be valid for higher dimensions.
}
 the formula for rotations of the spatial argument of a quaternion
module function $f$ subject to the QFT can be further simplified to
 \begin{gather}
  \widehat{f({\mathcal R}_{\vect{a}\vect{b}}\vect{x})}(\vect{u}) 
  \stackrel{\mbox{ in 2D }}{=}
  \hat{f}_-({\mathcal R}^{-1}_{\vect{a}\vect{b}}\vect{u})
      + \, \hat{f}_+
    ( {\mathcal R}_{\vect{a}\vect{b}}\vect{u})\, ,
  \label{eq:frot2D}
  \end{gather}
because in two dimensions we have 
${\mathcal U}_{\vect{e}_1} {\mathcal R}^{-1}_{\vect{a}\vect{b}}\,{\mathcal U}_{\vect{e}_1}
 ={\mathcal R}_{\vect{a}\vect{b}}\, .$

Theorems \ref{th:fpmtrafo} and \ref{th:GLtrafo} together with their clear 
geometric interpretation with the help of geometric algebra pave the way for
wide-ranging generalizations of the QFT of definition \ref{df:QFT}. In this
paper we cannot fully treat all possible generalizations. But in order to
demonstrate the method, we show in the following section 
how to generalize the QFT to a new
general non-commutative Fourier transformation of functions from 
spacetime $\R^{3,1}$ to the spacetime algebra \cite{DH:STA} of $\R^{3,1}$, i.e. to 
the Clifford geometric algebra $Cl_{3,1}\, .$
An intermediate step will be the generalization to a new Fourier transform
of functions from spacetime $\R^{3,1}$ to a volume-time subalgebra
of the spacetime algebra.

\section{Generalization of the QFT to a new spacetime algebra Fourier transform
         \label{sc:gen}}

We begin by recalling quaternion algebra to Clifford subalgebra isomorphisms 
such as $\HQ \cong Cl(0,2) \cong Cl^+(3,0)\, .$ 
Such isomorphisms together with the generalized $GL(\R^{n,m})$ 
transformation laws for 
$\{\hat{f}_{\pm}({\mathcal A}\vect{x})\}(\vect{u})$ allow us now to
generalize the QFT to higher dimensions. 

This indeed opens up a vast new field of related
\textit{multivector Fourier transforms}, which are in general
non-commutative.

\subsection{QFT generalization to volume-time functions \label{sc:VtFT}}

One of these quaternion algebra to Clifford sub-algebra isomorphisms 
that is of particular relevance in physics 
exists with a subalgebra of the \textit{spacetime algebra} $Cl_{3,1}.$ 
We express this isomorphism by introducing an
orthonormal (grade 1) vector basis for $\R^{3,1}$
\begin{equation}
  \{\vect{e}_0,\vect{e}_1, \vect{e}_2, \vect{e}_3  \},
  \quad
  -\vect{e}_0^2 = \vect{e}_1^2 = \vect{e}_2^2 = \vect{e}_3^2 = 1.
\end{equation}
Using this vector basis of $\R^{3,1},$ the spatial unit volume trivector $i_3$ 
and total four-dimensional (hyper volume) pseudoscalar $i_4$ can be expressed 
by
\begin{equation}
  i_3 = \vect{e}_1\vect{e}_2\vect{e}_3, \quad i_3^2 = -1,
  \quad
  i_4 = \vect{e}_0\vect{e}_1\vect{e}_2\vect{e}_3, \quad i_4^2=-1.
\end{equation}
We emphasize the fact that the vector $\vect{e}_0$, the 3D volume trivector
$i_3,$ and the 4D pseudoscalar $i_4,$ all square to minus one. Examining
the geometric algebra multiplication laws of $\vect{e}_0,$ $i_3,$ and $i_4,$
shows indeed that the arising subalgebra $V_t$ of the spacetime algebra is
isomorphic (see sections 4.1 and 4.2 of \cite{PG:quat}) to the quaternion algebra $\HQ$
\be
  V_t \cong \HQ \, ,
  \label{eq:QisomVt}
\end{equation}
where we use $V_t$ to denote the
volume-time subalgebra of $Cl_{3,1}$ with subalgebra basis 
\begin{equation}
  \{1, \vect{e}_0, i_3, i_4 \}.
\end{equation}
Note especially that
\be
  i_3 = \vect{e}_0 (-i_4) = \vect{e}_0 \,i_4^{-1} = \vect{e}_0^{*},
  \label{eq:i3dualt}
\end{equation}
which shows that $i_3$ is \textit{dual} to $\vect{e}_0$ in $Cl_{3,1}$.

Based on the isomorphism (\ref{eq:QisomVt}) we now define a Fourier transform
for volume-time module functions $ f \in L^2(\R^{3,1};V_t) $. 
\begin{defn}[Volume-time Fourier transform (VtFT)]
\label{df:SFT}
The volume-time Fourier transform 
$\overset{\circ}{f}: \R^{3,1} \rightarrow V_t $ of volume-time module functions
$ f \in L^2(\R^{3,1};V_t) $, with spacetime vectors
$\vect{x}= t\vect{e}_0 + \vec{x} \in \R^{3,1}, \,\,
 \vec{x} = x\vect{e}_1+y\vect{e}_2+z\vect{e}_3 \in \R^{3} $,
 and spacetime frequency vectors
$\vect{u}= s\vect{e}_0 + \vec{u} \in \R^{3,1}, \,\,
 \vec{u} = u\vect{e}_1+v\vect{e}_2+w\vect{e}_3 \in \R^{3} $ 
is defined as
\be
  \overset{\circ}{f}(\vect{u}) 
  = \int_{\R^{3,1}} e^{-\vect{e}_0 \,ts} f(\vect{x}) \,
    e^{-i_3 \vec{x}\cdot \vec{u}} d^4\vect{x} \,,
  \label{eq:sft}
\end{equation}
with the differential spacetime integration volume $d^4\vect{x} = dt dx dy dz \,.$
\end{defn}
The VtFT can be \textit{inverted} in close analogy to (\ref{eq:iqft}) by using
\be
  f(\vect{x}) 
  = \frac{1}{(2\pi)^4}\int_{\R^{3,1}} e^{\vect{e}_0 \, ts} \,
    \overset{\circ}{f}(\vect{u}) \,\,e^{i_3 \vec{x}\cdot \vec{u}} d^4\vect{u}\,,
\label{eq:isft}
\end{equation}
with $d^4\vect{u} = ds du dv dw \,.$

The $f_{\pm}$ split (\ref{eq:fpm}) combined with the isomorphism (\ref{eq:QisomVt}) 
now yields for volume-time module functions $ f \in L^2(\R^{3,1};V_t) $
\be
  f = f_+ + f_-, \quad
  f_+ = \frac{1}{2}(f + \vect{e}_0 f i_3), \quad 
  f_- = \frac{1}{2}(f - \vect{e}_0 f i_3).
\label{eq:fpms}
\end{equation}
Rewriting the split (\ref{eq:fpms}) with the duality relation (\ref{eq:i3dualt})
to
\be
  f_{\pm} = \frac{1}{2}(f \pm \vect{e}_0 f \vect{e}_0^{*})
\label{eq:fpmsd}
\end{equation}
shows that it naturally only depends on the physical spacetime split,
i.e. on the choice of the time direction $\vect{e}_0$.
Applying our new VtFT of definition \ref{df:SFT} to the split 
functions $f_{\pm}$ of (\ref{eq:fpmsd}) results in a VtFT formula 
which corresponds to theorem \ref{th:fpmtrafo}
\begin{gather}
  \overset{\circ}{f}_{\pm} 
  \stackrel{}{=} \int_{\R^{3,1}}
    f_{\pm}\,e^{-i_3 (\, \vec{x}\cdot \vec{u} \mp\, ts\,)}d^4x
  \stackrel{}{=} \int_{\R^{3,1}}
    e^{-\vect{e}_0 (\, ts \mp \vec{x}\cdot \vec{u}\,)}f_{\pm}\,d^4x \,.
  \label{eq:SFTfpm}
\end{gather}
Note especially that the $\overset{\circ}{f}_{+}$
part in (\ref{eq:SFTfpm}) has the kernel with the flat Minkowski 
metric $ts - \vec{x}\cdot \vec{u}$ in the exponent. (Compare section
\ref{sc:SplitInt} for further interpretation.)

Definition \ref{df:SFT} preserves the form of the $GL$ transformation properties of 
section \ref{sc:GLprop}. We get the $GL(\R^{3,1})$ transformation
properties of (\ref{eq:sft}) simply by inserting in theorem \ref{th:GLtrafo}
transformations ${\mathcal A} \in GL(\R^{3,1})$ and replacing
${\mathcal U}_{\vect{e}_1}$ by ${\mathcal U}_{\vect{e}_0}$. 
\begin{thm}[$GL(\R^{3,1})$ transformation properties of the VtFT]
\label{th:VtGLtrafo}
  The VtFT of a $V_t$ module function $f \in L^2(\R^2;V_t)$ with a 
  $GL(\R^{3,1})$ transformation ${\mathcal A}$ of its vector argument is given by 
  \begin{gather}
    \{f({\mathcal A}\vect{x})\}^{\circ}(\vect{u}) 
    = |\det {\mathcal A}^{-1}|\, \{\, \overset{\circ}{f}_-(\overline{{\mathcal A}^{-1}}\vect{u})
      + \, \overset{\circ}{f}_+
        ({\mathcal U}_{\vect{e}_0}\overline{{\mathcal A}^{-1}}\,
         {\mathcal U}_{\vect{e}_0}\vect{u}) \,\} \,\, .     
  \label{eq:VtGLtrafo}  
  \end{gather}
\end{thm}
In physical applications proper Lorentz transformations with $|\det {\mathcal A}|=1$
are most relevant, so the $|\det {\mathcal A}^{-1}|$ factor in
(\ref{eq:VtGLtrafo})  can then naturally be omitted.

For all kinds of applications it is of interest to know whether we can
push the QFT generalization established by the VtFT for 
volume-time module functions $ f \in L^2(\R^{3,1};V_t) $ even further, 
i.e. if even more general spacetime algebra functions can be treated
meaningfully with the VtFT. That this is indeed the case will be shown 
in the next subsection.

\subsection{Generalization to full spacetime algebra functions}

We now explain how we can drop in the VtFT definition
\ref{df:SFT} the restriction to volume-time functions
$ f \in L^2(\R^{3,1};V_t). $ 
The key to this is found in the commutativity of 
the unit volume trivector $i_3$ of the right side exponential factor
in (\ref{eq:sft})
with all spatial vectors $\{ \vect{e}_1, \vect{e}_2, \vect{e}_3 \}$
\be
  i_3 \,\vect{e}_k = \vect{e}_k\, i_3\,, \quad 1 \leq k \leq 3.
  \label{eq:i3comm}
\end{equation} 
This directly leads us to the \textit{right linearity} of the VtFT 
\begin{gather}
  \{f\alpha\}^{\circ}\,(\vect{u}) 
  = \int_{\R^{3,1}} e^{-\vect{e}_0 \,ts} f(\vect{x})\,\alpha \,\,
    e^{-i_3 \vec{x}\cdot \vec{u}} d^4\vect{x}  
  \nonumber \\
  = \int_{\R^{3,1}} e^{-\vect{e}_0 \,ts} f(\vect{x})\,\,
    e^{-i_3 \vec{x}\cdot \vec{u}} d^4\vect{x}\,\, \alpha  
  = \overset{\circ}{f}(\vect{u})\, \alpha,    
    \quad \forall \text{ const. } \alpha \in Cl_{3,0} ,
  \label{eq:sftlin}
\end{gather}
where $Cl_{3,0}$ is the eight-dimensional Clifford geometric algebra of $\R^{3,0}$, 
i.e. the 3D space subalgebra of $Cl_{3,1}$ spanned by
\be
  \{1, \vect{e}_1, \vect{e}_2, \vect{e}_3,  
    \vect{e}_2\vect{e}_3, \vect{e}_3\vect{e}_1, \vect{e}_1\vect{e}_2, i_3 \}.
\end{equation}
Naturally this right linearity also holds for the inverse transformation
\be
  f(\vect{x}) \, \alpha
  = \frac{1}{(2\pi)^4}\int_{\R^{3,1}} e^{\vect{e}_0 \, ts} \,
    \overset{\circ}{f}(\vect{u}) \alpha \,\,e^{i_3 \vec{x}\cdot \vec{u}} d^4\vect{u}
    \quad \forall \text{ const. } \alpha \in Cl_{3,0}.
\label{eq:isftlin}
\end{equation}
Now all 16 basis multivectors of $Cl_{3,1}$ can be obtained by successive geometric
multiplications of $1$ and $\vect{e}_0$ (or alternatively of $i_3$ and $i_4$, etc.)
with the three spatial vectors $\{ \vect{e}_1, \vect{e}_2, \vect{e}_3 \}$ from the 
right\footnote{%
$Cl_{3,1}$ is also isomorphic to the tensor product $V_t \otimes Cl_{3,0}^+$, with
$V_t$ defined as in section \ref{sc:VtFT} and $Cl_{3,0}^+$ defined as in section
\ref{sc:bquat}. (See \cite{PG:quat}, sections 4.1 and 4.2.)
}
\begin{gather}
  \{ 1, \vect{e}_1, \vect{e}_2, \vect{e}_3,  
    \vect{e}_2\vect{e}_3, \vect{e}_3\vect{e}_1, \vect{e}_1\vect{e}_2, i_3,
  \nonumber \\
    \vect{e}_0, \vect{e}_0\vect{e}_1, \vect{e}_0\vect{e}_2, \vect{e}_0\vect{e}_3,  
    \vect{e}_0\vect{e}_2\vect{e}_3, \vect{e}_0\vect{e}_3\vect{e}_1, 
    \vect{e}_0\vect{e}_1\vect{e}_2, \vect{e}_0i_3\}.
\end{gather}

We now have laid all the groundwork for the full spacetime algebra generalization of the
VtFT of definition \ref{df:SFT}:
\begin{defn}[Spacetime Fourier transform (SFT)]
\label{df:SFTgen}
  The SFT 
$\overset{\diamond}{f}: \R^{3,1} \rightarrow Cl_{3,1} $ 
of a (16 dimensional) spacetime algebra $Cl_{3,1}$ module function $f \in L^2(\R^{3,1};Cl_{3,1})$
with spacetime vectors
$\vect{x}= t\vect{e}_0 + \vec{x} \in \R^{3,1}, \,\,
 \vec{x} = x\vect{e}_1+y\vect{e}_2+z\vect{e}_3 \in \R^{3} $,
 and spacetime frequency vectors
$\vect{u}= s\vect{e}_0 + \vec{u} \in \R^{3,1}, \,\,
 \vec{u} = u\vect{e}_1+v\vect{e}_2+w\vect{e}_3 \in \R^{3} $ 
is defined by
\be
  \overset{\diamond}{f}(\vect{u}) 
  = \int_{\R^{3,1}} e^{-\vect{e}_0 \,ts} f(\vect{x}) \,
    e^{-i_3 \vec{x}\cdot \vec{u}} d^4\vect{x} \,,
  \label{eq:sftgen}
\end{equation}
with $d^4\vect{x} = dt dx dy dz \,.$
\end{defn}
Because of (\ref{eq:sftlin}) definition \ref{df:SFTgen} is fully compatible with 
definition \ref{df:SFT}, since
(\ref{eq:sftgen}) is nothing but a (right) linear combination of (\ref{eq:sft}). To
show this, we can use $Cl_{3,1} \cong V_t \otimes Cl_{3,0}^+$ or we can
e.g. rewrite a general spacetime algebra module 
function $f \in L^2(\R^{3,1};Cl_{3,1})$
as a (right) linear combination of four volume-time subalgebra module $L^2(\R^{3,1};V_t)$ 
functions  
\begin{gather}
  f = f_s + f_1 \vect{e}_1+ f_2 \vect{e}_2 +f_3 \vect{e}_3 
    +f_{23} \vect{e}_2\vect{e}_3 + f_{31} \vect{e}_3\vect{e}_1 + f_{12} \vect{e}_1\vect{e}_2
    +f_{123} i_3
    + f_{0} \vect{e}_0 +
  \nonumber \\
     + f_{01} \vect{e}_0\vect{e}_1 + f_{02} \vect{e}_0\vect{e}_2 
     + f_{03} \vect{e}_0\vect{e}_3 
     + f_{023} \vect{e}_0\vect{e}_2\vect{e}_3 + f_{031} \vect{e}_0\vect{e}_3\vect{e}_1
     + f_{012} \vect{e}_0\vect{e}_1\vect{e}_2  + f_{4}i_4
  \nonumber \\
    = f_s + f_{0} \vect{e}_0 + f_{123} i_3 + f_{4}i_4 
      + \{f_1 + f_{01} \vect{e}_0 + f_{23}i_3 + f_{023} i_4 \} \, \vect{e}_1 +
  \nonumber \\
      + \{f_2 + f_{02} \vect{e}_0 + f_{31}i_3 + f_{031} i_4 \} \, \vect{e}_2 
      + \{f_3 + f_{03} \vect{e}_0 + f_{12}i_3 + f_{012} i_4 \} \, \vect{e}_3.
  \label{eq:7Vtfs}
\end{gather}
The four $L^2(\R^{3,1};V_t)$ functions of (\ref{eq:7Vtfs}) are
$\{f_s + f_{0} \vect{e}_0 + f_{123} i_3 + f_{4}i_4,\,
   f_1 + f_{01} \vect{e}_0 + f_{23}i_3 + f_{023} i_4,\,
   f_2 + f_{02} \vect{e}_0 + f_{31}i_3 + f_{031} i_4,\,
   f_3 + f_{03} \vect{e}_0 + f_{12}i_3 + f_{012} i_4
\}$,
where all 16 coefficient functions $\{f_s, f_0, f_1, \cdots, f_4 \}$ belong to $L^2(\R^{3,1},\R)$.

Because of (\ref{eq:isftlin}) the general SFT of Clifford module $L^2(\R^{3,1};Cl_{3,1})$ functions
of definition \ref{df:SFTgen} is also invertible
\be
  f(\vect{x}) 
  = \frac{1}{(2\pi)^4}\int_{\R^{3,1}} e^{\vect{e}_0 \, ts} \,
    \overset{\diamond}{f}(\vect{u}) \,\,e^{i_3 \vec{x}\cdot \vec{u}} d^{\,4}\vect{u} \,\, .
\label{eq:gsftinv}
\end{equation}

\subsection{SFT of $f_{\pm}$ split parts and physical interpretation \label{sc:SplitInt}}

Further application of analogous (right) linearity arguments also yield that
the split (\ref{eq:fpms}) and (\ref{eq:fpmsd})  can also be applied to general multivector Clifford module 
functions $ f \in L^2(\R^{3,1};Cl_{3,1}) $. In (\ref{eq:fpms}) and (\ref{eq:fpmsd})  we can thus
simply replace the $L^2(\R^{3,1};V_t)$ functions by $L^2(\R^{3,1};Cl_{3,1}) $ 
functions\footnote{%
Again the $f_{\pm}$ split (\ref{eq:fpmsd}) solely depends on the choice of time direction 
$\vect{e}_0$.
}.
This carries on to the general SFTs of the split functions $f_{\pm}$, which are
formally identical to (\ref{eq:SFTfpm}) if we again replace 
the $L^2(\R^{3,1};V_t)$ functions by $L^2(\R^{3,1};Cl_{3,1}) $ functions. 

We can therefore rewrite the SFT (\ref{eq:sftgen}) for $f \in L^2(\R^{3,1};Cl_{3,1}) $ as
\begin{gather}
  \overset{\diamond}{f} = 
  \overset{\diamond}{f}_+ + \overset{\diamond}{f}_-
  \stackrel{}{=} \int_{\R^{3,1}} f_{+}\,e^{-i_3 (\, \vec{x}\cdot \vec{u} -\, ts\,)} d^4x
               + \int_{\R^{3,1}} f_{-}\,e^{-i_3 (\, \vec{x}\cdot \vec{u} +\, ts\,)}  d^4x
  \nonumber \\
  \stackrel{}{=} \int_{\R^{3,1}} e^{-\vect{e}_0 (\, ts - \vec{x}\cdot \vec{u}\,)}f_{+}\,d^4x
    + \int_{\R^{3,1}} e^{-\vect{e}_0 (\, ts + \vec{x}\cdot \vec{u}\,)}f_{-}\,d^4x \,.
  \label{eq:SFTfpmrcomb}
\end{gather}
Complex spacetime Fourier transformations, with 
$\exp\{-i(\vec{x}\cdot\vec{u}-ts)\}$ (where $i \in \C $) as the related complex kernel,
are e.g. used for electromagnetic fields in spatially dispersive media \cite{APW:emdispm}
or in electromagnetic wavelet theory \cite{GK:phwav}. 

In physics $f_+$ can be interpreted as (time dependent) multivector amplitude of a 
\textit{rightward} (forward) moving wave packet, and $f_-$ as that of a 
\textit{leftward} (backward)
moving wave packet. But we emphasize that both the non-commutative multivector
structure and the geometric interpretation (e.g. of $i_3$ as oriented 3D spatial volume
trivector) go beyond conventional treatment.

We get the consequent generalization of theorem \ref{th:GLtrafo},
i.e. the $GL(\R^{3,1})$ transformation properties of the SFT in the form of
\begin{thm}[$GL(\R^{3,1})$ transformation properties of the SFT]
\label{th:SGLtrafo}
  The SFT of a $Cl_{3,1}$ module function $f \in L^2(\R^2;Cl_{3,1})$ with a 
  $GL(\R^{3,1})$ transformation ${\mathcal A}$ of its vector argument is given by 
  \begin{gather}
    \{f({\mathcal A}\vect{x})\}^{\diamond}(\vect{u}) 
    = |\det {\mathcal A}^{-1}|\, \{\, \overset{\diamond}{f}_-(\overline{{\mathcal A}^{-1}}\vect{u})
      + \, \overset{\diamond}{f}_+
        ({\mathcal U}_{\vect{e}_0}\overline{{\mathcal A}^{-1}}\,
         {\mathcal U}_{\vect{e}_0}\vect{u}) \,\} \,\, .     
  \label{eq:SGLtrafo}  
  \end{gather}
\end{thm}

This concludes our brief example of a higher dimensional multivector generalization 
of the QFT for $L^2(\R^{2};\HQ)$ functions to a SFT for $L^2(\R^{3,1};Cl_{3,1}) $ functions. 
We again emphasize that mathematically many other generalizations are in fact possible and 
we expect a number of them to be of great utility in applications.

\section{Conclusions}

We employed a convenient rewriting of quaternions only in terms of $\qi$ and $\qj$, keeping
one to the left and the other to the right; and a quaternion split, 
which in spacetime applications is closely related the choice of the time direction. This allowed
us to investigate a range of properties of the QFT, last but not least the behavior
of the QFT under general linear automorphisms. 

General coordinate free formulation in 
combination with quaternion to Clifford subalgebra isomorphisms opens the door to a wide
range of QFT generalizations. These non-commutative multivector Fourier transforms act
on functions from $\R^{m,n}, \, m,n\in\N_0 $ to Clifford geometric algebras $Cl_{m,n}$
(or appropriate subalgebras). We demonstrated this by establishing two multivector
Fourier transforms: the volume-time and the spacetime Fourier transforms. They await
application, e.g. in the fields of dynamic fluid and gas flows, seismic analysis, 
to electromagnetic phenomena, in short wherever spatial data are recorded with time. 
We expect other generalizations of the QFT obtained by the same methods to
be of great potential use as well.

\section*{Acknowledgements}

I thank God, the Creator:
\textit{How great are your works, O LORD,
       how profound your thoughts!} \cite{NIV:Psalm92v5}.
I thank my family for their total loving support, 
and B. Mawardi, A. Hayashi and O. Yasukura for helpful comments.
I thank the editor S. Krausshar, and the anonymous referees for helpful
suggestions.

\end{document}